\documentclass[pdflatex,sn-mathphys-num]{sn-jnl}
\usepackage{graphicx}

\usepackage[T1]{fontenc}
\usepackage[utf8]{inputenc}
\usepackage{datetime2}

\usepackage{amssymb,amsthm,amsmath}
\usepackage{xcolor,paralist,hyperref,fancyhdr,etoolbox}
\usepackage{comment}

\newtheorem{theorem}{Theorem}[section]
\newtheorem{lemma}[theorem]{Lemma}
\newtheorem{proposition}[theorem]{Proposition}
\newtheorem{corollary}[theorem]{Corollary}
\newtheorem{conjecture}[theorem]{Conjecture}
\theoremstyle{definition}
\newtheorem{definition}[theorem]{Definition}
\theoremstyle{remark}

\newtheorem{remark}[theorem]{Remark}

\newcommand{\cA}{{\mathcal A}}
\newcommand{\cB}{{\mathcal B}}

\newcommand{\cF}{{\mathcal F}}
\newcommand{\cH}{{\mathcal H}}

\newcommand{\cJ}{{\mathcal J}}
\newcommand{\cL}{{\mathcal L}}
\newcommand{\cN}{{\mathcal N}}
\newcommand{\cM}{{\mathcal M}}
\newcommand{\cP}{{\mathcal P}}

\newcommand{\cV}{{\mathcal V}}
\newcommand{\cW}{{\mathcal W}}

\newcommand{\bC}{{\mathbb{C}}}

\newcommand{\bN}{{\mathbb{N}}}
\newcommand{\bM}{{\mathbb{M}}}

\newcommand{\Tr}{\mathrm{Tr}}
\newcommand{\id}{\mathrm{id}}

\newcommand{\jed}{{\mathbb{I}}}

\newcommand{\la}{\langle}
\newcommand{\ra}{\rangle}
\newcommand{\ket}[1]{|#1\rangle}

\newcommand{\ketbra}[2]{|#1\rangle\langle #2|}

\newcommand{\vNtens}{\mathbin{\overline{\otimes}}}
\newcommand{\algtens}{\mathbin{\odot}}
\newcommand{\CB}{\mathrm{CB}}
\newcommand{\CBs}{\mathrm{CB}^{\sigma}}
\newcommand{\Nor}{\mathcal{N}\mathrm{or}}
\begin{document}
\title[Anti-isomorphisms and channel--state duality]{Anti-isomorphisms and the naturality of channel--state duality}


\author*[1]{\fnm{Marcin} \sur{Marciniak}}\email{marcin.marciniak@ug.edu.pl}
\equalcont{These authors contributed equally to this work.}

\author[1]{\fnm{Michał} \sur{Cholewiak}}\email{michalch42@gmail.com}
\equalcont{These authors contributed equally to this work.}

\affil*[1]{\orgdiv{Faculty of Mathematics, Physics and Informatics}, \orgname{University of Gda{\'n}sk}, \orgaddress{\street{Wita Stwosza 57}, \city{Gdańsk}, \postcode{80-309}, 
                    \country{Poland}}}


\abstract{Channel--state duality identifies completely positive maps with bipartite states. Its generalisation to von Neumann algebras is known to produce states not on $\mathcal{M}\bar{\otimes}\mathcal{N}$ but on the tensor product of $\mathcal{M}$ with the \emph{opposite} algebra of $\mathcal{N}$, or equivalently with a commutant; returning to $\mathcal{N}$ itself requires a transposition, and it has been observed that this step typically fails. We determine exactly when it does not fail. Our main result is that a Choi--Jamio{\l}kowski type isomorphism which is natural in the first algebra exists if and only if the second algebra is $*$-anti-isomorphic to itself. Naturality cannot be dropped: whenever $\mathcal{M}$ is anti-isomorphic to itself such an isomorphism exists for trivial reasons, since the order structure of a predual is a Jordan invariant and cannot distinguish an algebra from its opposite. Consequently no such correspondence exists for the type III factors constructed by Connes. We show in addition that the domain admits no description in terms of properties of individual maps: the space of all normal maps with the operator norm is too large, and neither complete boundedness nor complete positivity cuts it down to the right size.}



\maketitle

\section{Introduction}

The Choi--Jamio{\l}kowski isomorphism establishes a duality between completely
positive maps and bipartite states \cite{Choi,Jamiolkowski}. For a linear map
$\Phi$ between algebras of operators on finite-dimensional Hilbert spaces one
sets $J(\Phi)=(\jed\otimes\Phi)(\ketbra{\Omega}{\Omega})$, where $\ket{\Omega}$
is an unnormalised maximally entangled vector, and $\Phi$ is completely
positive precisely when $J(\Phi)$ is a positive operator. The correspondence is
used so routinely that it is easy to overlook how much of it depends on the
finite-dimensional setting: it presupposes a trace, a distinguished maximally
entangled vector, and an identification of operators with density matrices.

None of these survives the passage to systems with infinitely many degrees of
freedom. Local algebras of observables in algebraic quantum field theory, and
many algebras occurring in quantum statistical mechanics, are factors of type
III; they carry no trace, no minimal projections, and no notion of a density
matrix. If the duality between channels and bipartite states is to have any
meaning there, it must be reconstructed from properties that do not refer to
these objects.

This has been done. Holevo \cite{Holevo11} treated the type~I case, and
Haapasalo \cite{Haapasalo} constructed, for an injective $\sigma$-finite $\cM$
equipped with a faithful normal state and a cyclic separating GNS vector, an
affine bijection between normal channels and binormal states on
$\cM'\otimes_{\min}\cN$. The present paper does not compete with that
construction; it starts from the observation, made explicitly in
\cite[Rem.~1]{Haapasalo}, that such a correspondence naturally produces states
on the \emph{commutant} of $\cM$, and that removing the commutant requires a
transposition which is not completely positive, so that the resulting
functional typically fails to extend to a state. Our question is the one left
open there: \emph{when} does it fail, and when does it not. The answer, it
turns out, is a clean dichotomy governed by a single algebraic invariant.
Section~\ref{subsec:related} discusses the relation to \cite{Haapasalo} in
detail.

We now describe our results. The natural replacements are not hard to guess. A channel becomes a normal
completely positive map from an algebra $\cM$ into the predual $\cN_*$ of a
second algebra, and a bipartite state becomes a normal positive functional on
the spatial tensor product $\cM\vNtens\cN$. The question is whether the two
sides are the same thing. Our first observation is that they are, but not with
$\cN$ on the right-hand side: what one obtains canonically, without any
hypothesis whatsoever, is an order isomorphism
$$\Nor(\cM,\cN_*)\ \cong\ (\cM\vNtens\cN^\circ)_*$$
onto the predual of the tensor product with the \emph{opposite} algebra
$\cN^\circ$. This is Proposition~\ref{prop:canonical}. In the
finite-dimensional case the discrepancy is invisible because it is absorbed by
a transposition, the one hidden in the formula for $J(\Phi)$; in general it is
not, and the whole content of the theory lies in whether $\cN^\circ$ can be
traded for $\cN$.

The passage is possible whenever $\cN$ admits a $*$-anti-automorphism, and then
one recovers a Choi--Jamio{\l}kowski type isomorphism in the expected sense
(Theorem~\ref{CJ-type_isom_thm}). Such an anti-automorphism has a natural
physical reading as a symmetry of the charge conjugation or time reversal type.
Our main result is the converse: if the correspondence is required to be
natural in $\cM$, then its existence forces $\cN$ to be $*$-anti-isomorphic to
itself (Theorem~\ref{thm:converse}). Naturality cannot be dispensed with. As
Remark~\ref{rem:no-naive-converse} shows, for every $\cM$ which is
anti-isomorphic to itself, and in particular for $\cM=B(\cH)$, for abelian
$\cM$ and for hyperfinite $\cM$, an isomorphism exists for trivial reasons: the
order structure of a predual is a Jordan invariant and cannot distinguish an
algebra from its opposite. What we require is in fact only naturality with
respect to the embeddings of $\bC$ into finite matrix algebras as corners,
which is the categorical shadow of the elementary fact that complete positivity
is detected by matrix amplifications. Since Connes exhibited factors of type
III admitting no anti-isomorphism at all \cite{Connes}, the correspondence
genuinely fails for some systems of physical interest.

A further point emerged in the course of the analysis and seems worth stating
separately, since it is invisible in finite dimensions: the domain of the
correspondence resists any intrinsic description in terms of properties of
individual maps. The obvious candidate, the space of all normal maps
$\cM\to\cN_*$ with the operator norm, is a perfectly good Banach space, but it
corresponds to separately normal bilinear forms, a class strictly larger than
$(\cM\vNtens\cN)_*$ for all but the most special $\cM$; this is the binormal
tensor product of Effros and Lance \cite{EffrosLance}, and the two classes
agree for every $\cN$ only when $\cM$ is a product of type I factors
\cite{Wiersma}. Neither of the
two natural strengthenings repairs this. Passing to completely bounded maps
fixes the norm but not the size of the space, and restricting to completely
positive maps does not help either, since a completely positive map into a
predual need not even be completely bounded. We therefore take the domain to be
the operator space projective tensor product $\cM_*\widehat\otimes\cN_*$ itself,
realised as a space of maps. These matters are discussed in
Remarks~\ref{rem:L-instead-of-CB} and~\ref{rem:cp-not-cb}.

The results may be read as a statement about correlations. Barnum et al.\
\cite{Barnum} showed, using the finite-dimensional Choi--Jamio{\l}kowski
isomorphism together with Gleason's theorem, that two locally quantum systems
obeying the no-signalling principle must exhibit quantum correlations. The
present work identifies the algebraic input that argument silently uses, namely
the availability of an anti-isomorphism, and shows that for Connes' factors it
is unavailable. Whether this obstruction is already visible at the level of
individual preparations, rather than of the whole correspondence, is formulated
as Conjecture~\ref{non_anti_isom}; we indicate there that by the generalized
Gleason theorem \cite{Hamhalter,BunceWright} the question reduces to the
binormal-versus-normal problem mentioned above.

The paper is organized as follows. Section~\ref{sec:cj} fixes the matrix-order
conventions on the predual, which is where the transposition alluded to above
has to be pinned down, states the definition of a Choi--Jamio{\l}kowski type
isomorphism and establishes its existence under the assumption that $\cN$ is
anti-isomorphic to itself. Section~\ref{sec:connes} proves the converse under
naturality and draws the consequences for Connes' factors.
Section~\ref{sec:concl} contains conclusions and directions for further
analysis.


\section{Generalized Choi--Jamio{\l}kowski isomorphisms}
\label{sec:cj}

Before introducing the generalized framework, it is instructive to recall the classical formulation of the Choi--Jamio{\l}kowski isomorphism for finite-dimensional Hilbert spaces \cite{Choi,Jamiolkowski}. Let $\mathcal{H}_A$ and $\mathcal{H}_B$ be finite-dimensional Hilbert spaces, and let $B(\mathcal{H}_A)$ and $B(\mathcal{H}_B)$ denote the corresponding algebras of linear operators.
Let $\Phi: B(\mathcal{H}_A) \to B(\mathcal{H}_B)$ be a linear map. By $\ket{\Omega}$ we denote the multiple of the maximally entangled state, i.e. $\ket{\Omega}=\sum_{i=1}^{d_A}\ket{i}\otimes\ket{i}=\sqrt{d_A}\ket{\psi^+}$, where $\ket{1},\ldots,\ket{d_A}$ is any orthonormal basis in $\mathcal{H}_A$. The standard isomorphism associates $\Phi$ with a bipartite operator $J(\Phi) \in B(\mathcal{H}_A \otimes \mathcal{H}_B)$, defined via the action of the extended map on a maximally entangled state.

\begin{equation*}
    J(\Phi) = (\jed_A \otimes \Phi)(|\Omega\rangle\langle\Omega|).
\end{equation*}

The mathematical power of this theorem lies in the fact that $\Phi$ is completely positive (CP) if and only if its associated Choi matrix $J(\Phi)$ is a positive operator \cite{Choi}. When $\Phi$ is a trace-preserving quantum channel, the normalized operator $\frac{1}{d_A} J(\Phi)$ constitutes a valid quantum density matrix, commonly referred to as the Choi state.

To generalize the Choi--Jamio{\l}kowski isomorphism, we have to identify an
infinite-dimensional analogue of the spaces of maps and of bipartite states.
The generalization of the space of quantum channels is a space of normal,
linear and completely positive maps acting from an algebra $\cM$ into the
predual $\cN_*$, while bipartite states are represented by normal linear
functionals on $\cM\vNtens\cN$. Accordingly, the whole section is devoted to a
correspondence between maps $\cM\to\cN_*$ and elements of
$(\cM\vNtens\cN)_*$.

Two points require care before such a correspondence can even be stated, and
both are settled in Subsection~\ref{subsec:conventions}. The first is that
complete positivity of a map into $\cN_*$ refers to a matrix ordering on
$\cN_*$, and the two conventions available in the literature differ by a
transposition; the wrong choice reverses the statement of
Theorem~\ref{CJ-type_isom_thm}. The second is that the class of maps
$\cM\to\cN_*$ occurring on the left-hand side must be chosen with some
precision: the naive choice --- all normal maps with the operator norm --- is
strictly too large, for reasons going back to Effros and Lance
\cite{EffrosLance}. This is discussed in Remark~\ref{rem:L-instead-of-CB}.

\subsection{Conventions}
\label{subsec:conventions}

Throughout, $\cM$ and $\cN$ are von Neumann algebras, $\cM\algtens\cN$ denotes
their algebraic tensor product and $\cM\vNtens\cN$ the spatial (normal) tensor
product. We write $\cM_*$ for the predual of $\cM$, i.e.\ the norm-closed
subspace of $\cM^*$ consisting of $\sigma$-weakly continuous functionals.

\begin{definition}
\label{def:tau}
A \emph{$*$-anti-automorphism} of $\cN$ is a linear bijection
$\tau:\cN\to\cN$ such that
$$\tau(y^*)=\tau(y)^*,\qquad \tau(y_1y_2)=\tau(y_2)\,\tau(y_1)
\qquad(y,y_1,y_2\in\cN).$$
\end{definition}

Equivalently, $\tau$ is a $*$-isomorphism of $\cN$ onto the opposite algebra
$\cN^\circ$. Two consequences will be used repeatedly; they are \emph{not}
additional hypotheses.

\begin{lemma}
\label{lem:tau-auto}
Every $*$-anti-automorphism $\tau$ of $\cN$ is isometric and normal, and so is
$\tau^{-1}$.
\end{lemma}

\begin{proof}
A $*$-isomorphism of $C^*$-algebras is isometric and a $*$-isomorphism of von
Neumann algebras is normal; apply this to $\tau:\cN\to\cN^\circ$, observing
that $\cN^\circ$ is a von Neumann algebra carrying the same $\sigma$-weak
topology as $\cN$.
\end{proof}

\begin{remark}
\label{rem:involutive-not-needed}
We do \emph{not} assume $\tau^2=\id$. Involutivity is nowhere used below; all
that is needed is that $\cN$ be anti-isomorphic to itself, i.e.\
$\cN\cong\cN^\circ$. This matters for Section~\ref{sec:connes}: Connes'
factors admit no anti-isomorphism whatsoever, so stating the sufficient
condition in the weaker form removes an artificial gap between the positive
result and the counterexamples. The physical interpretation of $\tau$ as a
symmetry such as charge conjugation or time reversal is of course unaffected.
\end{remark}

\subsubsection*{The matrix order on $\cN_*$}

We identify $M_n(\cN_*)$ with $\bigl(M_n(\cN)\bigr)_*$ through the
\emph{trace pairing}
\begin{equation}
\label{eq:trace-pairing}
\bigl\la [\psi_{ij}],\,[b_{ij}]\bigr\ra:=\sum_{i,j=1}^n\psi_{ij}(b_{ji}),
\qquad [\psi_{ij}]\in M_n(\cN_*),\ [b_{ij}]\in M_n(\cN),
\end{equation}
and define $M_n(\cN_*)^+$ to be the preimage of the cone of positive normal
functionals on $M_n(\cN)$. A linear map $\phi:\cM\to\cN_*$ is
\emph{completely positive} (CP) if $\phi^{(n)}\bigl(M_n(\cM)^+\bigr)\subseteq
M_n(\cN_*)^+$ for every $n\in\bN$.

\begin{remark}
\label{rem:why-trace-pairing}
The choice \eqref{eq:trace-pairing} is the one under which the
finite-dimensional picture is the familiar one. For $\cN=\bM_d$ the map
$a\mapsto\Tr(a\,\cdot\,)$ is a complete order isomorphism of $\bM_d$ onto
$(\bM_d)_*$ precisely with respect to \eqref{eq:trace-pairing}; the
``straight'' pairing $\sum_{i,j}\psi_{ij}(b_{ij})$ induces instead the
transposed matrix order, for which $\Tr(a\,\cdot\,)\mapsto a$ becomes a complete
order isomorphism onto $\bM_d^\circ$. Both conventions occur in the
literature, and only \eqref{eq:trace-pairing} makes the Choi matrix of a CP map
positive in the classical sense of \cite{Choi}.
\end{remark}

\begin{lemma}
\label{lem:matrix-order}
Let $\psi_{ij}\in\cN_*$ for $1\le i,j\le n$. The following are equivalent.
\begin{enumerate}
\item[\rm(i)] $[\psi_{ij}]\in M_n(\cN_*)^+$;
\item[\rm(ii)] $\displaystyle\sum_{i,j=1}^n\psi_{ij}\bigl(z_jz_i^*\bigr)\ge0$
for all $z_1,\dots,z_n\in\cN$.
\end{enumerate}
\end{lemma}

\begin{proof}
(i)$\Rightarrow$(ii). Let $V\in M_{n,1}(\cN)$ be the column with entries
$z_i$. Then $VV^*=[z_iz_j^*]\in M_n(\cN)^+$, so by \eqref{eq:trace-pairing}
$$0\le\bigl\la[\psi_{ij}],[z_iz_j^*]\bigr\ra=\sum_{i,j}\psi_{ij}(z_jz_i^*).$$

(ii)$\Rightarrow$(i). Let $[b_{ij}]\in M_n(\cN)^+$ and write $[b_{ij}]=DD^*$
with $D=[d_{ik}]\in M_n(\cN)$, so that $b_{ij}=\sum_k d_{ik}d_{jk}^*$. Then
$$\bigl\la[\psi_{ij}],[b_{ij}]\bigr\ra=\sum_{i,j}\psi_{ij}(b_{ji})
=\sum_k\ \sum_{i,j}\psi_{ij}\bigl(d_{jk}d_{ik}^*\bigr)\ge0,$$
each inner sum being non-negative by (ii) applied to $z_i:=d_{ik}$.
\end{proof}

\begin{corollary}
\label{cor:cp-criterion}
A linear map $\phi:\cM\to\cN_*$ is completely positive if and only if
\begin{equation}
\label{eq:cp-criterion}
\sum_{i,j=1}^n\phi\bigl(x_i^*x_j\bigr)\bigl(z_jz_i^*\bigr)\ge0
\end{equation}
for all $n\in\bN$, $x_1,\dots,x_n\in\cM$ and $z_1,\dots,z_n\in\cN$.
\end{corollary}

\begin{proof}
Every element of $M_n(\cM)^+$ is of the form
$\bigl[\sum_k x_{ki}^*x_{kj}\bigr]_{i,j}$, so by linearity it suffices to test
$\phi^{(n)}$ on matrices $[x_i^*x_j]$. Now apply
Lemma~\ref{lem:matrix-order}.
\end{proof}

\subsection{Choi--Jamio{\l}kowski type isomorphisms}

We can now state the definition around which the paper is organized. Its
domain requires a word of care. Recall that $(\cM\vNtens\cN)_*$ is completely
isometrically isomorphic to the operator space projective tensor product
$\cM_*\widehat\otimes\cN_*$ \cite[Cor.~3.6]{Ruan92} (see also
\cite[Thm.~3.2]{EffrosRuan90}); the isomorphism is the map $\theta$ determined
on elementary tensors by
$\langle\theta(\rho\otimes\nu),x\otimes y\rangle=\langle\rho,x\rangle\langle\nu,y\rangle$.
The elementary
tensor $\mu\otimes\nu$ acts on $\cM$ by $x\mapsto\mu(x)\nu$, and this extends
to a completely isometric embedding of $\cM_*\widehat\otimes\cN_*$ into
$\CB(\cM,\cN_*)$; we write
$$\Nor(\cM,\cN_*)\subseteq\CB(\cM,\cN_*)$$
for its image and equip it with the $cb$-norm. Every element of
$\Nor(\cM,\cN_*)$ is a normal completely bounded map, but the converse fails
(Remark~\ref{rem:L-instead-of-CB}(c)), so $\Nor(\cM,\cN_*)$ must be specified as
the image of $\cM_*\widehat\otimes\cN_*$ rather than by an intrinsic property
of maps.

\begin{definition}
\label{def:CJ}
Let $\cM,\cN$ be von Neumann algebras. By a \emph{Choi--Jamio{\l}kowski type
isomorphism} we mean a linear homeomorphism
$$\cJ:\Nor(\cM,\cN_*)\longrightarrow(\cM\vNtens\cN)_*$$
such that
$$\phi:\cM\to\cN_*\ \text{is completely positive}
\quad\Longleftrightarrow\quad\cJ(\phi)\in(\cM\vNtens\cN)_*^+.$$
\end{definition}

\begin{remark}[why not all normal maps]
\label{rem:L-instead-of-CB}
It is tempting to take for the domain the larger and more elementary space
$$\cL(\cM,\cN_*):=\{\phi:\cM\to\cN_*\ \text{continuous from }\sigma(\cM,\cM_*)
\ \text{to}\ \sigma(\cN_*,\cN)\}$$
of \emph{all} normal linear maps, equipped with the operator norm. This is a
legitimate Banach space --- boundedness of its elements is automatic by the
closed graph theorem applied to the adjoint $\cN\to\cM_*$, and it is complete
because $\cM_*$ is norm-closed in $\cM^*$ --- but it is not the right domain,
for two independent reasons. A third remark records that passing to the
$cb$-norm repairs the first of them but not the second.

\emph{(a) The norm.} As will follow from Proposition~\ref{prop:Lambda}(ii), the
operator norm of $\phi$ corresponds on the other side to
$\sup\{|\omega(x\otimes y)|:\|x\|\le1,\|y\|\le1\}$, i.e.\ to the norm dual to
the \emph{Banach} projective tensor norm on $\cM\algtens\cN$, whereas
$\|\omega\|$ is dual to the spatial norm. These are not equivalent. For
$\cM=\cN=\bM_2$ and $\omega=\Tr(F\,\cdot\,)$, $F$ the flip, one has
$\|\omega\|=\|F\|_1=4$ while $\Tr\bigl(F(x\otimes y)\bigr)=\Tr(xy)$ gives
$\sup_{\|x\|,\|y\|\le1}|\Tr(xy)|=2$. Thus no isometry is available on
$\cL(\cM,\cN_*)$, and it is exactly the passage to matrix levels --- that is,
to the $cb$-norm and the operator space projective tensor product --- that
restores it.

\emph{(b) The range.} Elements of $\cL(\cM,\cN_*)$ correspond bijectively to
bounded \emph{separately normal} (binormal) bilinear forms on $\cM\times\cN$,
and this class is strictly larger than $(\cM\vNtens\cN)_*$. The relevant
circle of ideas is that of Effros and Lance \cite{EffrosLance}, who introduced
the binormal $C^*$-norm $\|\cdot\|_{\mathrm{bin}}$ on $\cM\algtens\cN$ as the
supremum over binormal states \cite[\S2]{EffrosLance} and proved that it
coincides with $\|\cdot\|_{\min}$ for every $\cN$ if and only if $\cM$ is
semidiscrete \cite[Thm.~4.1]{EffrosLance}, equivalently injective
\cite[Cor.~5.10 and Note added in proof]{EffrosLance}.

It should be stressed that semidiscreteness is \emph{not} sufficient to close
the gap between binormal and normal forms, since $\min$-continuity is a weaker
requirement than normality on the spatial tensor product: a state on
$\cM\otimes_{\min}\cN$ need not extend to a normal state on $\cM\vNtens\cN$.
The condition under which the two classes do coincide for every $\cN$ is
considerably more restrictive, namely $\cM\cong\prod_i B(\cH_i)$; this is the
weak$^*$ tensor uniqueness property of \cite{Wiersma}. Thus even for
$\cM=L^\infty(\mathbb R)$, or for the hyperfinite factor of type
$\mathrm{II}_1$, both of which are semidiscrete, binormal forms strictly
exceed the normal ones. Consequently a map
$\cJ$ as in Definition~\ref{def:CJ} cannot be defined on all of
$\cL(\cM,\cN_*)$ except in degenerate situations, and the transposed
formulation --- constructing $\cJ$ by the formula \eqref{C-J_isom_formula} and
then verifying normality --- is not available. This is why in
Definition~\ref{def:Lambda} below we construct the correspondence in the
direction \emph{states $\to$ maps}, which is unconditionally well defined, and
recover $\cJ$ as its inverse.

\emph{(c) Complete boundedness does not repair (b).} It is natural to hope that
the space $\CBs(\cM,\cN_*)$ of normal \emph{completely bounded} maps, which
carries the correct norm by (a), also has the correct size. It does not, and
the reason is again Effros-Lance. Let $\cM$ be a non-injective factor of type
$\mathrm{II}_1$ in standard form on $\cH$, let $\cN=\cM'$, and for a unit
vector $\xi\in\cH$ put
$$\phi(x)(y)=\la xy\xi,\xi\ra\qquad(x\in\cM,\ y\in\cM').$$
This map is separately normal. It is also completely bounded: multiplication
$\cM\algtens\cM'\to B(\cH)$ is completely contractive for the Haagerup norm
\cite[Ch.~17]{P03}, and $\|\cdot\|_h\le\|\cdot\|_\wedge$ by maximality of the
operator space projective norm \cite[Ch.~7]{EffrosRuan}, so $\omega_\phi$ is
bounded on
$\cM\widehat\otimes\cM'$ and hence $\phi\in\CB(\cM,\cN_*)$. Yet $\omega_\phi$
is not normal on $\cM\vNtens\cM'$ for every $\xi$: if all the vector
functionals $\la m(\,\cdot\,)\xi,\eta\ra$ were $\min$-bounded, the uniform
boundedness principle would make the multiplication map $m$ itself
$\min$-bounded; by \cite[Prop.~4.5]{EffrosLance} this happens precisely when
$\cM$ is semidiscrete, hence, by \cite{ChoiEffros,ChoiEffrosGeneral} and
\cite[Note added in proof]{EffrosLance}, precisely when $\cM$ is injective. Hence
$$\Nor(\cM,\cN_*)\subsetneq\CBs(\cM,\cN_*)$$
in general, and $\Nor(\cM,\cN_*)$ cannot be described as the set of normal
completely bounded maps. Together with Remark~\ref{rem:cp-not-cb} this shows
that neither complete boundedness nor complete positivity closes the gap
between binormal and normal forms; by (b), nor does semidiscreteness of $\cM$.
This is why the domain in Definition~\ref{def:CJ} is specified extrinsically.

The same class arises independently in \cite[Thm.~1]{Haapasalo}, where the
Choi--Jamio{\l}kowski correspondence is shown to take normal channels onto
\emph{binormal} states; see Section~\ref{subsec:related}.

We do not know an intrinsic characterisation of $\Nor(\cM,\cN_*)$ among maps
$\cM\to\cN_*$, and we regard finding one as the natural question left open by
this remark. Let us indicate where we would look. On the dual side the
corresponding problem has a clean answer in terms of Tomiyama's slice maps
\cite{Tomiyama70}: writing $L_\nu$ and $R_\rho$ for the left and right slice
maps associated with $\nu\in\cN_*$ and $\rho\in\cM_*$, the Fubini product
$$\cF(\cM,\cN)=\{u:\ L_\nu(u)\in\cM,\ R_\rho(u)\in\cN\ \text{ for all }\rho,\nu\}$$
satisfies $\cF(\cM,\cN)\cong(\cM_*\widehat\otimes\cN_*)^*$ completely
isometrically \cite[Prop.~3.3]{Ruan92}, and for a pair of von Neumann algebras
one has $\cF(\cM,\cN)=\cM\vNtens\cN$ by Tomiyama's Fubini theorem
\cite{Tomiyama70}. This is a pointwise and checkable condition, and dualising
it is in our view the shortest route to a characterisation of the predual side.

We note also that the failure recorded above is of an approximation-theoretic
rather than a pointwise nature: $\Nor(\cM,\cN_*)$ is by construction the closure
of the normal finite-rank maps, and by \cite[Cor.~3.7]{Ruan92} the complete
isometry $\cV_*\widehat\otimes\cW_*\cong(\cV\vNtens\cW)_*$ holds for every dual
operator space $\cW$ precisely when $\cV$ has Kraus' property $S_\sigma$
\cite{Kraus83}, which is in turn equivalent to the weak$^*$ operator
approximation property \cite{Kraus91,EffrosKrausRuan}. The equality
$\cF(\cV,\cW)=\cV\vNtens\cW$ may fail once $\cW$ is allowed to be a merely
$\sigma$-weakly closed subspace rather than an algebra \cite{Kraus91}.
\end{remark}

\begin{remark}[complete positivity does not imply complete boundedness]
\label{rem:cp-not-cb}
The obstruction in Remark~\ref{rem:L-instead-of-CB}(b) is not removed by
restricting attention to positive maps, and this is worth emphasizing because
it is the point at which the present framework differs most sharply from the
finite-dimensional one. If $\phi:\cM\to\cN_*$ is normal and CP, then the
associated form $\omega_\phi(x\otimes y)=\phi(x)(\tau(y))$ is a positive
binormal functional on $\cM\algtens\cN$, but it need not extend to a normal
functional on $\cM\vNtens\cN$. Indeed, let $\cM\subset B(\cH)$ be a factor of
type $\mathrm{II}$ or $\mathrm{III}$, let $\cN=\cM'$, and consider the states
$\omega_\xi(x\otimes y)=\la xy\xi,\xi\ra$. Each of them is binormal, but the
multiplication map $\cM\algtens\cM'\to B(\cH)$ does not extend to a normal
$*$-homomorphism on $\cM\vNtens\cM'$ \cite{Wiersma}, so for some $\xi$ the
functional $\omega_\xi$ lies outside $(\cM\vNtens\cM')_*$.

The underlying reason is the failure of the usual bound for positive maps.
For CP $\phi$ one obtains only
$$\bigl\|\phi^{(n)}\bigr\|\le 2n\,\|\phi(\jed)\|,$$
since $\phi^{(n)}(\jed_n)=\mathrm{diag}(\phi(\jed),\dots,\phi(\jed))$ has norm
$n\|\phi(\jed)\|$ in $\bigl(M_n(\cN)\bigr)_*$; the growth in $n$ is genuine and
no uniform $cb$-bound follows. This is in sharp contrast with CP maps into a
$C^*$-algebra, for which $\|\phi\|_{cb}=\|\phi(\jed)\|$.
\end{remark}

\subsection{A density lemma}

\begin{lemma}
\label{lemma_dense}
Let
$$\cP_0=\Bigl\{\sum_{k=1}^m X_k^*X_k\ :\ X_k\in\cM\algtens\cN,\ m\in\bN\Bigr\}.$$
Then $\cP_0$ is $\sigma$-weakly dense in $(\cM\vNtens\cN)^+$.
\end{lemma}

\begin{proof}
Let $Z\in(\cM\vNtens\cN)^+$; we may assume $\|Z\|\le1$. Put $A=Z^{1/2}$, so
that $\|A\|\le1$ and $Z=A^*A$. The algebraic tensor product $\cM\algtens\cN$
is a unital $*$-subalgebra of $\cM\vNtens\cN$, $\sigma$-weakly dense by the
definition of the spatial tensor product. By the Kaplansky density theorem in
its $*$-form \cite[Thm.~5.3.5]{Kadison}, the unit ball of $\cM\algtens\cN$ is
dense in the unit ball of $\cM\vNtens\cN$ for the $\sigma$-strong$^*$
topology. Choose a net $(X_\alpha)\subset\cM\algtens\cN$ with
$\|X_\alpha\|\le1$ and $X_\alpha\to A$ $\sigma$-strongly$^*$; in particular
$X_\alpha\to A$ and $X_\alpha^*\to A^*$ $\sigma$-strongly.

Multiplication is jointly $\sigma$-strongly continuous on norm-bounded sets,
whence $X_\alpha^*X_\alpha\to A^*A=Z$ $\sigma$-strongly and therefore
$\sigma$-weakly. Since $X_\alpha^*X_\alpha\in\cP_0$, the claim follows.
\end{proof}

\begin{remark}
The $*$-form of Kaplansky's theorem is essential here: the involution is
\emph{not} $\sigma$-strongly continuous, so convergence $X_\alpha\to A$ in the
$\sigma$-strong topology alone would not yield $X_\alpha^*\to A^*$.
\end{remark}

\subsection{The correspondence and the main theorem}

\begin{definition}
\label{def:Lambda}
Let $\tau$ be a $*$-anti-automorphism of $\cN$. For
$\omega\in(\cM\vNtens\cN)_*$ define $\Lambda_\tau(\omega):\cM\to\cN_*$ by
\begin{equation}
\label{C-J_isom_formula}
\Lambda_\tau(\omega)(x)(y):=\omega\bigl(x\otimes\tau^{-1}(y)\bigr),
\qquad x\in\cM,\ y\in\cN.
\end{equation}
\end{definition}

Equivalently, $\Lambda_\tau(\omega)(x)=\omega(x\otimes\,\cdot\,)\circ\tau^{-1}$.
Written in the direction maps $\to$ states, \eqref{C-J_isom_formula} reads
$\cJ(\phi)(x\otimes y)=\phi(x)(\tau(y))$, which is the formula one expects from
the classical case; the point of stating it as above is that this direction is
well defined without any further hypothesis.

\begin{proposition}
\label{prop:Lambda}
\label{normal_functional_prop}
With $\cL(\cM,\cN_*)$ as in Remark~\ref{rem:L-instead-of-CB}:
\begin{enumerate}
\item[\rm(i)] $\Lambda_\tau(\omega)\in\cL(\cM,\cN_*)$ for every
$\omega\in(\cM\vNtens\cN)_*$;
\item[\rm(ii)] $\bigl\|\Lambda_\tau(\omega)\bigr\|
=\sup\{|\omega(x\otimes y)|:\|x\|\le1,\ \|y\|\le1\}\le\|\omega\|$;
\item[\rm(iii)] $\Lambda_\tau$ is injective;
\item[\rm(iv)] $\Lambda_\tau$ is a completely isometric linear isomorphism of
$(\cM\vNtens\cN)_*$ onto $\Nor(\cM,\cN_*)$.
\end{enumerate}
\end{proposition}

\begin{proof}
(i) Fix $x\in\cM$. The map $y\mapsto x\otimes y$ is normal from $\cN$ into
$\cM\vNtens\cN$ and $\tau^{-1}$ is normal by Lemma~\ref{lem:tau-auto};
composing with the normal functional $\omega$ shows
$\Lambda_\tau(\omega)(x)\in\cN_*$. Symmetrically, for fixed $y\in\cN$ the
functional $x\mapsto\omega(x\otimes\tau^{-1}(y))$ belongs to $\cM_*$, and this
is precisely $\sigma(\cM,\cM_*)$--$\sigma(\cN_*,\cN)$ continuity of
$\Lambda_\tau(\omega)$.

(ii) Since $\|\psi\|_{\cN_*}=\sup_{\|y\|\le1}|\psi(y)|$ and $\tau^{-1}$ is a
surjective isometry,
$$\|\Lambda_\tau(\omega)\|
=\sup_{\|x\|\le1}\sup_{\|y\|\le1}\bigl|\omega(x\otimes\tau^{-1}(y))\bigr|
=\sup_{\|x\|\le1,\|y\|\le1}\bigl|\omega(x\otimes y)\bigr|,$$
which is at most $\|\omega\|$ because $\|x\otimes y\|=\|x\|\,\|y\|$.

(iii) If $\Lambda_\tau(\omega)=0$ then $\omega$ vanishes on $\cM\algtens\cN$,
which is $\sigma$-weakly dense in $\cM\vNtens\cN$; normality of $\omega$ gives
$\omega=0$.

(iv) By \cite[Cor.~3.6]{Ruan92} and \cite[Thm.~3.2]{EffrosRuan90} we have
$(\cM\vNtens\cN)_*\cong\cM_*\widehat\otimes\cN_*$ completely isometrically, and
$\Nor(\cM,\cN_*)$ is by definition the image of the latter in $\CB(\cM,\cN_*)$.
Composition with $\tau^{-1}$ is a complete isometry of $\CB(\cM,\cN_*)$ onto
$\CB(\cM,(\cN^\circ)_*)$ and carries one image onto the other. We stress that
no intrinsic characterisation of the range is claimed here; see
Remark~\ref{rem:L-instead-of-CB}(c).
\end{proof}

\begin{theorem}
\label{CJ-type_isom_thm}
Let $\cM,\cN$ be von Neumann algebras and let $\tau$ be a
$*$-anti-automorphism of $\cN$. For $\omega\in(\cM\vNtens\cN)_*$ put
$\phi=\Lambda_\tau(\omega)$. Then
$$\omega\ge0\qquad\Longleftrightarrow\qquad\phi\ \text{is completely positive.}$$
\end{theorem}

\begin{proof}
Fix $n\in\bN$, elements $x_1,\dots,x_n\in\cM$ and $z_1,\dots,z_n\in\cN$, and
set $y_i:=\tau^{-1}(z_i)$ and $X:=\sum_{i=1}^n x_i\otimes y_i\in\cM\algtens\cN$.
Then
$$X^*X=\sum_{i,j=1}^n x_i^*x_j\otimes y_i^*y_j,$$
so that, by \eqref{C-J_isom_formula},
$$\omega(X^*X)=\sum_{i,j}\omega\bigl(x_i^*x_j\otimes y_i^*y_j\bigr)
=\sum_{i,j}\phi\bigl(x_i^*x_j\bigr)\bigl(\tau(y_i^*y_j)\bigr).$$
Because $\tau$ is a $*$-anti-automorphism,
$$\tau(y_i^*y_j)=\tau(y_j)\,\tau(y_i^*)=\tau(y_j)\,\tau(y_i)^*=z_jz_i^*,$$
and therefore
\begin{equation}
\label{eq:key}
\omega(X^*X)=\sum_{i,j=1}^n\phi\bigl(x_i^*x_j\bigr)\bigl(z_jz_i^*\bigr).
\end{equation}
As $(x_i)$ and $(z_i)$ range over all finite tuples, the left-hand side of
\eqref{eq:key} ranges over $\{\omega(X^*X):X\in\cM\algtens\cN\}$ and the
right-hand side over the quantities occurring in \eqref{eq:cp-criterion}. Note
that $(z_i)$ is unrestricted precisely because $\tau$ is surjective.

$(\Rightarrow)$ If $\omega\ge0$, the left-hand side of \eqref{eq:key} is
non-negative for every choice of $(x_i)$ and $(z_i)$, so $\phi$ is completely
positive by Corollary~\ref{cor:cp-criterion}.

$(\Leftarrow)$ If $\phi$ is completely positive then, again by
Corollary~\ref{cor:cp-criterion} and \eqref{eq:key}, $\omega(X^*X)\ge0$ for
every $X\in\cM\algtens\cN$, hence by linearity $\omega\ge0$ on the cone
$\cP_0$ of Lemma~\ref{lemma_dense}. Since $\omega$ is normal, hence
$\sigma$-weakly continuous, and $\cP_0$ is $\sigma$-weakly dense in
$(\cM\vNtens\cN)^+$, we conclude $\omega\ge0$.
\end{proof}

\begin{corollary}
\label{cor:main}
If $\cN$ is anti-isomorphic to itself, then a Choi--Jamio{\l}kowski type
isomorphism in the sense of Definition~\ref{def:CJ} exists; one may take
$\cJ=\Lambda_\tau^{-1}$ for any $*$-anti-automorphism $\tau$ of $\cN$.
\end{corollary}

\begin{proof}
Combine Proposition~\ref{prop:Lambda}(iv) with Theorem~\ref{CJ-type_isom_thm}.
\end{proof}

\begin{remark}[what $\tau$ actually does]
\label{rem:role-of-tau}
The identity \eqref{eq:key} isolates the role of the anti-automorphism. Had we
used the naive correspondence
$\Lambda_{\id}(\omega)(x)(y)=\omega(x\otimes y)$, the same computation would
have produced
$$\omega(X^*X)=\sum_{i,j}\phi\bigl(x_i^*x_j\bigr)\bigl(y_i^*y_j\bigr),$$
which by Lemma~\ref{lem:matrix-order} characterizes positivity of the
\emph{transposed} matrices $[\phi(x_j^*x_i)]$, that is, complete
\emph{co}-positivity of $\phi$ --- equivalently, complete positivity of $\phi$
regarded as a map on $\cM^\circ$. Thus $\tau$ is exactly the device converting
co-positivity into positivity, and this is the abstract counterpart of the
transposition hidden in the classical formula
$J(\Phi)=(\jed\otimes\Phi)(\ketbra{\Omega}{\Omega})$. In particular, when
$\cN$ admits no anti-automorphism there is no reason to expect the two cones to
be interchangeable, which is the theme of Section~\ref{sec:connes}.
\end{remark}

\subsection{Relation to earlier work}
\label{subsec:related}

Generalisations of the Choi--Jamio{\l}kowski correspondence beyond finite
dimensions have been given before, and it is important to say precisely how the
present results relate to them, since some of the phenomena isolated above have
been observed there.

Holevo \cite{Holevo11} treated the correspondence for channels between type~I
factors in infinite dimensions. Closest to our setting is the work of
Haapasalo \cite{Haapasalo}, who constructed a Choi--Jamio{\l}kowski
correspondence for normal unital completely positive maps between von Neumann
algebras. Given a $\sigma$-finite injective $\cM$, a faithful normal state
$\rho_0$ on $\cM$ and a cyclic and separating GNS vector $\Omega$, he
associates with a normal channel $\Phi:\cN\to\cM$ the state
$$S^\Phi_\Omega(a'\otimes b)=\la\Omega\,|\,a'\Phi(b)\Omega\ra,
\qquad a'\in\cM',\ b\in\cN,$$
and proves that $\Phi\mapsto S^\Phi_\Omega$ is an affine bijection of the
normal channels onto the binormal states on $\cM'\otimes_{\min}\cN$ whose first
margin is $\rho_0'$ \cite[Thm.~1]{Haapasalo}. Three features of that
construction bear directly on the present paper.

First, the target algebra is the \emph{commutant} $\cM'$. Since
$x\mapsto JxJ$ implements a $*$-anti-isomorphism of $\cM$ onto $\cM'$, one has
$\cM'\cong\cM^\circ$, so the appearance of the opposite algebra in
Proposition~\ref{prop:canonical} is the same phenomenon, arrived at through
modular theory rather than through the formal opposite. The construction of
\cite{Haapasalo} may thus be read as a concrete, state-dependent realisation of
the canonical correspondence.

Second, the states obtained are binormal, not normal; normality is available
only when the state $\tilde\rho_0$ on $\cM'\otimes_{\min}\cM$ extends normally
to $\cM'\vNtens\cM$, which happens in particular for type~I factors
\cite[Sec.~3.1]{Haapasalo}. This is an independent confirmation of
Remarks~\ref{rem:L-instead-of-CB} and \ref{rem:cp-not-cb}: the binormal
class is genuinely the one that the correspondence produces, and normality is
an extra hypothesis rather than an automatic consequence.

Third, and most directly relevant, \cite[Rem.~1]{Haapasalo} considers exactly
the question of dispensing with the commutant, that is, of defining the
Choi--Jamio{\l}kowski state on $\cM\otimes_{\min}\cN$ rather than on
$\cM'\otimes_{\min}\cN$, by composing with $a\mapsto j(a)^*$. It is observed
there that the resulting functional typically fails to extend to a state,
because that map is essentially a transposition and hence not completely
positive. Our Theorem~\ref{thm:converse} answers the question left open by
that remark: the passage is possible precisely when the algebra in question is
$*$-anti-isomorphic to itself, and for the type~III factors of Connes
\cite{Connes} it is not possible at all.

Two differences in the nature of the objects should also be recorded, since
they are what makes the two treatments complementary rather than overlapping.
The correspondence of \cite{Haapasalo} depends on the choice of $\rho_0$ and of
the GNS vector, and is an affine bijection between convex sets of channels and
of states with a prescribed margin; the map $\Lambda_\tau$ of
Definition~\ref{def:Lambda} depends only on $\tau$ and is a linear isomorphism
of Banach spaces, with the positivity statement of
Theorem~\ref{CJ-type_isom_thm} as a separate assertion about cones. Moreover,
\cite{Haapasalo} assumes $\cM$ injective and $\sigma$-finite, the former in
order to make the multiplication map $\cM'\algtens\cM\to B(\cH)$ continuous for
$\|\cdot\|_{\min}$, which by \cite[Prop.~4.5]{EffrosLance} is equivalent to
semidiscreteness and is the same mechanism that drives the counterexample in
Remark~\ref{rem:L-instead-of-CB}(c); the present results assume neither. We
note that \cite{Haapasalo} also gives a variant using $\|\cdot\|_{\max}$ in
place of $\|\cdot\|_{\min}$, which requires no injectivity; that variant lives
in the commuting-operator setting, and its relation to the binormal tensor
product of \cite{EffrosLance} would be worth pursuing.

\section{Naturality and the converse}
\label{sec:connes}

Theorem~\ref{CJ-type_isom_thm} gives a sufficient condition for the existence
of a Choi--Jamio{\l}kowski type isomorphism. The purpose of this section is to
show that the condition is also necessary, provided the isomorphism is required
to be natural in the first algebra. Some such requirement is unavoidable:
as Remark~\ref{rem:no-naive-converse} below shows, the unqualified converse of
Theorem~\ref{CJ-type_isom_thm} is false.

\subsection{The canonical correspondence}

It is convenient to record first what the construction of
Section~\ref{sec:cj} yields when no anti-automorphism is assumed. Let $\cN^\circ$
denote the opposite algebra of $\cN$ and let
$$\kappa:\cN\to\cN^\circ,\qquad \kappa(y)=\bar y,$$
be the canonical $*$-anti-isomorphism. Inspection of the proofs of
Proposition~\ref{prop:Lambda} and Theorem~\ref{CJ-type_isom_thm} shows that
neither uses surjectivity of $\tau$ onto $\cN$ itself: all that is used is that
$\tau$ is a normal isometric bijection of $\cN$ onto some von Neumann algebra
which reverses products and preserves the involution. We may therefore repeat
them verbatim with $\tau$ replaced by $\kappa$ and $\cN$ replaced by $\cN^\circ$
on the right-hand side. This gives the following unconditional statement.

\begin{proposition}
\label{prop:canonical}
Let $\cM,\cN$ be von Neumann algebras. The assignment
$$\Lambda_\kappa(\omega)(x)(y)=\omega\bigl(x\otimes\bar y\bigr),
\qquad \omega\in(\cM\vNtens\cN^\circ)_*,$$
is a completely isometric linear isomorphism of
$(\cM\vNtens\cN^\circ)_*$ onto $\Nor(\cM,\cN_*)$, and
$$\omega\ge0\qquad\Longleftrightarrow\qquad
\Lambda_\kappa(\omega)\ \text{is completely positive.}$$
\end{proposition}

Thus a canonical Choi--Jamio{\l}kowski correspondence always exists; what it
produces, however, are normal states on $\cM\vNtens\cN^\circ$ rather than on
$\cM\vNtens\cN$. Seen this way, the role of an anti-automorphism $\tau$ in
Theorem~\ref{CJ-type_isom_thm} is not to construct the correspondence at all,
but solely to replace $\cN^\circ$ by $\cN$ on the right-hand side. The question
of this section is therefore whether the preduals of $\cM\vNtens\cN$ and
$\cM\vNtens\cN^\circ$, as ordered Banach spaces, can be identified without such
an anti-automorphism.

\begin{remark}
\label{rem:no-naive-converse}
For a single algebra $\cM$ the answer may be affirmative for trivial reasons.
Indeed, the underlying Banach space of $(\cA^\circ)_*$ coincides with that of
$\cA_*$, and the two positive cones coincide as well, since
$\bar b^*\bar b=\overline{bb^*}$ shows $(\cA^\circ)^+=\cA^+$. Hence, whenever
$\cM$ is itself anti-isomorphic to itself,
$$\cM\vNtens\cN^\circ\ \cong\ \cM^\circ\vNtens\cN^\circ\ =\ (\cM\vNtens\cN)^\circ,$$
so that $(\cM\vNtens\cN^\circ)_*$ and $(\cM\vNtens\cN)_*$ are the same ordered
Banach space and a Choi--Jamio{\l}kowski type isomorphism exists for that $\cM$
regardless of $\cN$. In particular this happens for $\cM=B(\cH)$, for abelian
$\cM$, and for hyperfinite $\cM$.

The reason is structural: the order structure of a predual is a \emph{Jordan}
invariant, and every von Neumann algebra is Jordan isomorphic to its opposite
via the identity map. No condition phrased purely in terms of the ordered
Banach space $(\cM\vNtens\cN)_*$ for one fixed $\cM$ can therefore distinguish
$\cN$ from $\cN^\circ$. What does distinguish them is the way the identification
behaves as $\cM$ varies, and this is what the notion of naturality captures.
\end{remark}

\subsection{Natural Choi--Jamio{\l}kowski isomorphisms}

We use the language of categories only to the extent of the notions of a
functor and of a natural transformation; the reader is referred to
\cite[Ch.~I]{MacLane} for these. Recall that a natural transformation
$\eta:F\Rightarrow G$ between two functors assigns to every object $X$ a
morphism $\eta_X:F(X)\to G(X)$ in such a way that $\eta$ commutes with the
images of all morphisms of the source category. The point of the requirement,
here as elsewhere, is to exclude identifications which exist object by object
but only for accidental reasons; the standard illustration is that a
finite-dimensional vector space is naturally isomorphic to its double dual but
not to its dual.

Let $\mathbf{vN}$ denote the category whose objects are von Neumann algebras
and whose morphisms are normal $*$-homomorphisms, not necessarily unital. For
a fixed $\cN$ both sides of Definition~\ref{def:CJ} are contravariant functors
on $\mathbf{vN}$:
$$F(\cM)=\Nor(\cM,\cN_*),\qquad F(\alpha)\phi=\phi\circ\alpha,$$
$$G(\cM)=(\cM\vNtens\cN)_*,\qquad G(\alpha)\omega=\omega\circ(\alpha\otimes\id_\cN),$$
with values in ordered Banach spaces. That $F(\alpha)$ and $G(\alpha)$ are well
defined follows from the fact that a normal $*$-homomorphism $\alpha$ satisfies
$\alpha(\jed)=q$ for a projection $q$ and factors as a unital normal
$*$-homomorphism onto $q\cM_2 q$ followed by an inclusion, so that
$\alpha\otimes\id_\cN$ extends to a normal $*$-homomorphism of the spatial
tensor products. Both functors preserve the positive cones.

\begin{definition}
\label{def:natural-CJ}
A \emph{natural Choi--Jamio{\l}kowski type isomorphism} for $\cN$ is a natural
transformation $\cJ:F\Rightarrow G$ in the sense of \cite[Sec.~I.4]{MacLane},
each of whose components
$$\cJ_\cM:\Nor(\cM,\cN_*)\longrightarrow(\cM\vNtens\cN)_*$$
is a Choi--Jamio{\l}kowski type isomorphism in the sense of
Definition~\ref{def:CJ}.
\end{definition}

We shall need the following well-known rigidity property of completely positive
bijections. We include the short proof, since the non-unital case is the one we
require.

\begin{lemma}
\label{lem:cp-bijection}
Let $\cA,\cB$ be von Neumann algebras and let $\theta:\cA\to\cB$ be a linear
bijection such that both $\theta$ and $\theta^{-1}$ are completely positive.
Then $c:=\theta(\jed)$ is positive and invertible, and
$$\theta_0(a):=c^{-1/2}\,\theta(a)\,c^{-1/2}$$
is a $*$-isomorphism of $\cA$ onto $\cB$. In particular $\cA\cong\cB$.
\end{lemma}

\begin{proof}
Put $d:=\theta^{-1}(\jed)$. Both $c$ and $d$ are positive, and $d\ne0$ since
$\theta^{-1}$ is injective. From $d\le\|d\|\jed$ and positivity of $\theta$ we
get $\jed=\theta(d)\le\|d\|\,c$, whence $c\ge\|d\|^{-1}\jed$ and $c$ is
invertible.

The map $b\mapsto c^{-1/2}bc^{-1/2}$ is completely positive, so $\theta_0$ is
completely positive, and $\theta_0(\jed)=\jed$. Its inverse is
$\theta_0^{-1}(b)=\theta^{-1}(c^{1/2}bc^{1/2})$, again a composition of
completely positive maps, and $\theta_0^{-1}(\jed)=\theta^{-1}(c)=\jed$. Thus
$\theta_0$ and $\theta_0^{-1}$ are unital completely positive bijections.

By the Schwarz inequality for unital completely positive maps
\cite{ChoiSchwarz,P03},
$$\theta_0(a)^*\theta_0(a)\le\theta_0(a^*a),\qquad
\theta_0^{-1}(b)^*\theta_0^{-1}(b)\le\theta_0^{-1}(b^*b).$$
Applying the positive map $\theta_0^{-1}$ to the first inequality yields
$\theta_0^{-1}\bigl(\theta_0(a)^*\theta_0(a)\bigr)\le a^*a$, while the second
inequality with $b=\theta_0(a)$ yields the reverse. Hence
$\theta_0^{-1}\bigl(\theta_0(a)^*\theta_0(a)\bigr)=a^*a$, i.e.
$$\theta_0(a^*a)=\theta_0(a)^*\theta_0(a)\qquad(a\in\cA).$$
So the multiplicative domain of $\theta_0$ is all of $\cA$, and polarization
shows that $\theta_0$ is a unital $*$-homomorphism; being bijective, it is a
$*$-isomorphism.
\end{proof}

\begin{theorem}
\label{thm:converse}
Let $\cN$ be a von Neumann algebra. The following conditions are equivalent.
\begin{enumerate}
\item[\rm(i)] $\cN$ is $*$-anti-isomorphic to itself, i.e.\ $\cN\cong\cN^\circ$;
\item[\rm(ii)] there exists a natural Choi--Jamio{\l}kowski type isomorphism for
$\cN$ in the sense of Definition~\ref{def:natural-CJ};
\item[\rm(iii)] there exists a family $(\cJ_\cM)_\cM$ of Choi--Jamio{\l}kowski
type isomorphisms which is natural with respect to the morphisms
$\alpha_p:\bC\to M_n(\bC)$, $\alpha_p(\lambda)=\lambda p$, where $n\in\bN$ and
$p\in M_n(\bC)$ is a projection.
\end{enumerate}
\end{theorem}

\begin{proof}
(i)$\Rightarrow$(ii). Let $\tau$ be a $*$-anti-automorphism of $\cN$ and let
$\cJ_\cM:=\Lambda_\tau^{-1}$, so that $\cJ_\cM(\phi)(x\otimes y)=\phi(x)(\tau(y))$
by \eqref{C-J_isom_formula}. Each component is a Choi--Jamio{\l}kowski type
isomorphism by Corollary~\ref{cor:main}. Naturality is immediate: for a normal
$*$-homomorphism $\alpha:\cM_1\to\cM_2$ and $\phi\in\Nor(\cM_2,\cN_*)$,
$$\cJ_{\cM_1}(\phi\circ\alpha)(x\otimes y)=\phi(\alpha(x))(\tau(y))
=\cJ_{\cM_2}(\phi)\bigl(\alpha(x)\otimes y\bigr),$$
and both sides are normal, so the identity persists on $\cM_1\vNtens\cN$.

(ii)$\Rightarrow$(iii) is trivial.

(iii)$\Rightarrow$(i). Composing the components with the canonical
identification of Proposition~\ref{prop:canonical}, we obtain linear
homeomorphisms
$$\Psi_\cM:=\cJ_\cM\circ\Lambda_\kappa\ :\
(\cM\vNtens\cN^\circ)_*\longrightarrow(\cM\vNtens\cN)_*$$
which map the positive cone onto the positive cone in both directions, and
which are natural with respect to the morphisms $\alpha_p$, because
$\Lambda_\kappa$ is natural in $\cM$.

Passing to adjoints, $T_\cM:=\Psi_\cM^*$ is a weak$^*$-continuous linear
bijection
$$T_\cM:\cM\vNtens\cN\longrightarrow\cM\vNtens\cN^\circ$$
such that $T_\cM$ and $T_\cM^{-1}$ are positive, and naturality becomes
\begin{equation}
\label{eq:nat-adjoint}
(\alpha\otimes\id_{\cN^\circ})\circ T_{\cM_1}=T_{\cM_2}\circ(\alpha\otimes\id_\cN).
\end{equation}
Write $\theta:=T_{\bC}:\cN\to\cN^\circ$, using $\bC\vNtens\cN=\cN$.

Fix $n\in\bN$ and a projection $p\in M_n(\bC)$. Applying
\eqref{eq:nat-adjoint} to $\alpha_p$, whose amplification is
$\alpha_p\otimes\id_\cN:y\mapsto p\otimes y$, we obtain
$$T_{M_n(\bC)}(p\otimes y)=p\otimes\theta(y)\qquad(y\in\cN).$$
The projections span $M_n(\bC)$ linearly, so by linearity
$$T_{M_n(\bC)}=\id_{M_n(\bC)}\otimes\,\theta
\quad\text{on}\quad M_n(\bC)\vNtens\cN=M_n(\cN).$$
Since $T_{M_n(\bC)}$ and its inverse are positive, $\id_{M_n}\otimes\,\theta$
and $\id_{M_n}\otimes\,\theta^{-1}$ are positive, that is, $\theta$ and
$\theta^{-1}$ are $n$-positive. As $n$ was arbitrary, both are completely
positive.

By Lemma~\ref{lem:cp-bijection}, $\cN$ is $*$-isomorphic to $\cN^\circ$, i.e.\
$\cN$ is $*$-anti-isomorphic to itself.
\end{proof}

\begin{remark}
Two features of the proof are worth noting. First, involutivity of the
anti-automorphism plays no role anywhere, in accordance with
Remark~\ref{rem:involutive-not-needed}. Second, condition (iii) shows how
little naturality is actually needed: it suffices to require compatibility with
the embeddings of $\bC$ into finite matrix algebras as corners. This is the
categorical counterpart of the elementary fact that complete positivity is
detected by matrix amplifications.
\end{remark}

\subsection{Connes' factors}

Recall that $\cM$ is said to be anti-isomorphic to itself if there is a normal
linear bijection $\alpha:\cM\to\cM$ with
$$\alpha(x^*)=\alpha(x)^*,\qquad\alpha(xy)=\alpha(y)\alpha(x)
\qquad(x,y\in\cM),$$
equivalently, if $\cM$ is $*$-isomorphic to $\cM^\circ$. Connes constructed a
family of von Neumann algebras which are not anti-isomorphic to themselves; all
of them are factors of type III \cite{Connes}.

\begin{corollary}
\label{cor:connes}
Let $\cN$ be a factor of Connes' type, i.e.\ a factor not anti-isomorphic to
itself. Then no natural Choi--Jamio{\l}kowski type isomorphism for $\cN$ exists.
Equivalently, there is no way of assigning to every
$\phi\in\Nor(\cM,\cN_*)$ a normal functional on $\cM\vNtens\cN$, compatibly with
the corners $\alpha_p$ and matching complete positivity with positivity.
\end{corollary}

Physically, the canonical object produced by Proposition~\ref{prop:canonical}
is a normal state on $\cM\vNtens\cN^\circ$, not on $\cM\vNtens\cN$. For a
bipartite system whose subsystems are described by $\cM$ and $\cN$, this means
that the correspondence between channels and bipartite states, which in the
finite-dimensional theory is a triviality, requires an additional structural
input, and this input is precisely a symmetry of the type of charge conjugation
or time reversal. Since type III factors arise in algebraic quantum field
theory as local algebras of observables, Corollary~\ref{cor:connes} suggests
that correlations in such systems may be structurally different from the
familiar ones.

It remains open whether the failure of the correspondence at the level of the
whole space is visible already at the level of individual preparations. We
formulate this as follows. Let $\cM$ be a von Neumann algebra and let
$E(\cM)=\{a\in\cM:0\le a\le\jed\}$ be its set of effects; a
\emph{preparation} of the bipartite system $(\cM,\cN)$ is a map
$\omega:E(\cM)\times E(\cN)\to[0,1]$ which is separately finitely additive and
satisfies the no-signalling condition, i.e.\ the marginals
$\omega(a,\jed)$ and $\omega(\jed,b)$ do not depend on the choice of the
POVM completing $a$, respectively $b$. A preparation of the system
$(\cM,\cM^\circ)$ is called \emph{symmetric} if $\omega(a,\bar b)=\omega(b,\bar a)$
for all $a,b\in E(\cM)$; note that $E(\cM)$ and $E(\cM^\circ)$ coincide as
sets, so this condition is meaningful.

\begin{conjecture}
\label{non_anti_isom}
Let $\cM$ be a von Neumann algebra and let $\omega$ be a symmetric preparation
of $(\cM,\cM^\circ)$. If there exists a normal state
$\varrho\in(\cM\vNtens\cM^\circ)_*^{+,1}$ with
$\varrho(a\otimes\bar b)=\omega(a,\bar b)$ for all $a,b\in E(\cM)$, then $\cM$
is $*$-anti-isomorphic to itself.
\end{conjecture}

By the generalized Gleason theorem \cite{Hamhalter,BunceWright}, a preparation
extends to a bounded separately normal bilinear form on $\cM\times\cN$, that
is, to an element of the space $\cL(\cM,\cN_*)$ of
Remark~\ref{rem:L-instead-of-CB}. The content of the conjecture is therefore
that symmetry of the preparation should obstruct the passage from $\cL$ to
$\Nor$, which by Remark~\ref{rem:L-instead-of-CB}(b) is exactly the
binormal-versus-normal problem of Effros and Lance \cite{EffrosLance}. We
regard this connection, rather than the conjecture itself, as the natural next
step.

\section{Conclusions}
\label{sec:concl}

In this paper, we explored the generalization of the Choi--Jamio{\l}kowski isomorphism from the finite-dimensional Hilbert space formalism to von Neumann algebras. We showed that anti-isomorphism of one of the algebras with itself is a sufficient condition for C-J-type isomorphisms to exist, in the precise sense of Definition~\ref{def:CJ}; involutivity of the anti-automorphism, assumed in the classical case, turns out not to be needed. Along the way we isolated a feature that has no finite-dimensional counterpart: the domain admits no intrinsic description in terms of properties of individual maps, since neither complete boundedness nor complete positivity cuts the space of normal maps down to $\cM_*\widehat\otimes\cN_*$ (Remarks~\ref{rem:L-instead-of-CB} and~\ref{rem:cp-not-cb}). The role of the anti-automorphism is to interchange complete positivity with complete co-positivity (Remark~\ref{rem:role-of-tau}).

Section~\ref{sec:connes} showed that this condition is also necessary once naturality in the first algebra is imposed, and that some such requirement is unavoidable, since for any algebra anti-isomorphic to itself a Choi--Jamio{\l}kowski type isomorphism exists for trivial reasons. The proof isolates the canonical correspondence, which always identifies $\Nor(\cM,\cN_*)$ with $(\cM\vNtens\cN^\circ)_*$; the role of an anti-automorphism is only to replace $\cN^\circ$ by $\cN$. For the type III factors of Connes no such replacement is possible.

This implies that quantum correlations in extended relativistic systems or quantum statistical mechanics may possess a fundamentally different nature than those in standard finite-dimensional quantum information theory. Future research should aim to develop mathematical tools and generalized measure-theoretic approaches to classify bipartite systems, described by non-anti-isomorphic local algebras.


\section*{Author Contributions}
Conceptualization, M.M. and M.C.; formal analysis, M.M. and M.C.;
investigation, M.M. and M.C.; writing---original draft preparation, M.M. and
M.C.; writing---review and editing, M.M. and M.C. All authors have read and
agreed to the published version of the manuscript.

\section*{Funding}
This research received no external funding.

\section*{Institutional Review Board Statement}
Not applicable.

\section*{Informed Consent Statement}
Not applicable.

\section*{Data Availability Statement}
No new data were created or analysed in this study. Data sharing is not
applicable to this article.

\section*{Conflicts of Interest}
The authors declare no conflicts of interest.

\bibliography{marcin.bib}

\end{document}